\documentclass[11pt,a4paper]{amsart}
\usepackage{a4wide}
\usepackage{mathrsfs, amssymb,amsmath,url,amsthm,color,comment}
\usepackage{hyperref}
\usepackage[dvipsnames]{xcolor}
\usepackage{todonotes}

\newcommand\new[1]{{#1}}

\title{Cores over Ramsey  structures}

\author
{Antoine Mottet}
	\address{Department of Algebra, MFF UK, Sokolovsk\'a 83, 186 00 Praha 8, Czech Republic}
	\email{mottet@karlin.mff.cuni.cz}
	\urladdr{http://www.karlin.mff.cuni.cz/~mottet/}

\author
{Michael Pinsker}
	\address{Institut f\"{u}r Diskrete Mathematik und Geometrie, FG Algebra, TU Wien, Austria, and Department of Algebra, Charles University, Czech Republic}    
	\email{marula@gmx.at}
    \urladdr{http://dmg.tuwien.ac.at/pinsker/}

\thanks{
Antoine Mottet has received funding from the European Research Council
(ERC) under the European Unions Horizon 2020 research and
innovation programme (grant agreement No 771005). Michael Pinsker has received funding from the  Austrian Science Fund (FWF) through  project No P32337, and from the Czech Science Foundation (grant No 18-20123S)}

\theoremstyle{plain}

    \newtheorem{theorem}{Theorem}
    \newtheorem{lemma}[theorem]{Lemma}

    \newtheorem{definition}[theorem]{Definition}
    
     \newtheorem{question}[theorem]{Question}

 \newcommand{\ignore}[1]{}

\DeclareMathOperator{\Aut}{Aut}
\DeclareMathOperator{\End}{End}
\DeclareMathOperator{\Emb}{Emb}

\newcommand{\sA}{\mathbb A}
\newcommand{\sB}{\mathbb B}
\newcommand{\sC}{\mathbb C}
\newcommand{\sS}{\mathbb S}
\newcommand{\sF}{\mathbb F}

\newcommand{\fM}{\mathscr M}
\newcommand{\fN}{\mathscr N}
\newcommand{\fG}{\mathscr G}
\newcommand{\fS}{\mathscr S}

\newcommand{\cK}{\mathcal K}
\newcommand{\cF}{\mathcal F}

\usepackage{todonotes}

\begin{document}

\maketitle

\begin{abstract}
We prove that if an $\omega$-categorical structure has an $\omega$-categorical homogeneous Ramsey expansion, then so does its model-complete core.
\end{abstract}

\section{Introduction}

\subsection{Model-complete cores} A countable  $\omega$-categorical structure $\sA$ is \emph{model-complete} if the automorphisms of $\sA$ form a dense subset of the self-embeddings of $\sA$ in the topology of pointwise convergence. 
It is a \emph{core} if all its endomorphisms are self-embeddings.
A classical result by Saracino~\cite{Saracino} states that every $\omega$-categorical structure $\sA$ is \emph{embedding-equivalent} to a model-complete structure:  there exists a model-complete structure $\sA'$ such that $\sA$ embeds into $\sA'$ and vice-versa. The structure $\sA'$ is moreover unique to isomorphism, and is called the \emph{model-companion} of $\sA$. 
This result was later subsumed by a result of~\cite{Cores-journal},
where it is proved that every $\omega$-categorical structure $\sA$ is \emph{homomorphically equivalent} to a model-complete core  $\sA'$, i.e., $\sA$ has a homomorphism to $\sA'$ and vice-versa. Again, $\sA'$ is unique up to isomorphism, and called the \emph{model-complete core  of $\sA$}. The latter result was inspired by and has gained considerable importance in applications to \emph{constraint satisfaction problems}. These are certain computational problems associated with relational structures which are invariant under homomorphic equivalence. Much of the algebraic theory of such problems only works for model-complete cores -- see~\cite{BKOPP,BKOPP-equations, wonderland, BartoPinskerDichotomy, Topo}.

One weakness of the proof in~\cite{Cores-journal} as well as of the later proof in~\cite{BKOPP-equations} is that they are non-constructive, and it is often hard to decide what properties  are transferred from $\sA$ to its model-complete core $\sA'$. One important property which \emph{is} known to be preserved is $\omega$-categoricity~\cite{Cores-journal}, and so are \emph{homogeneity} and \emph{finite boundedness}, as can be seen easily from the proofs. 
However, often the structure  $\sA$ of interest does not have these properties itself, but is only first-order definable in a structure $\sB$ which might enjoy them; we then call $\sA$ a \emph{first-order reduct} of $\sB$. It is interesting to note here that the model-complete core $\sA'$ of a structure $\sA$ is sensitive to which of the first-order definable relations of $\sA$ belong to its signature; in particular, even if $\sA$ and $\sB$ are first-order reducts of one another, their model-complete cores might differ. And while the above-mentioned properties are inherited by $\sA'$ from $\sA$  rather straightforwardly, it is less clear whether, say, \new{the model-complete core} of a first-order reduct of a finitely bounded homogeneous structure is again a first-order reduct of a finitely bounded homogeneous structure. This question is of fundamental importance in applications to  constraint satisfaction problems mentioned above, in particular in relation with a complexity dichotomy conjecture of Bodirsky and the second author (see~\cite{BPP-projective-homomorphisms}).

\begin{question}\label{quest:main}
Is the class of first-order reducts of a finitely bounded homogeneous structures closed under taking model-complete cores?
\end{question}

\subsection{Ramsey expansions}
In this article, we shall answer this question for first-order reducts of finitely bounded homogeneous  \emph{Ramsey} structures. This additional assumption might, however, be void since the following problem is open.


\begin{question}\label{quest:Ramsey}
Does every finitely bounded  homogeneous structure have a finitely bounded homogeneous Ramsey expansion? In other words, is every first-order reduct of a finitely bounded homogeneous structure also a  first-order reduct of a finitely bounded  homogeneous Ramsey structure?
\end{question}

This question has, in various formulations and with varying scope, been considered by several authors: by  Bodirsky and the second author when formulating the dichotomy conjecture, and also in the context of a  decidability result for first-order reducts of finitely bounded homogeneous Ramsey structures~\cite{BPT-decidability-of-definability}; and later by  Melleray, Van Th\'{e}, and Tsankov~\cite{MNT15-Polish} in the context and language  of topological dynamics. In~\cite{BodirskyRamsey}, Bodirsky formulates a conjecture claiming a positive answer for homogeneous structures in a finite language,  and extensively argues the importance of this conjecture.  Ne\v{s}et\v{r}il  addressed the question indirectly in the context of the characterisation of Ramsey classes~\cite{RamseyClasses}. Hubi\v{c}ka and Ne\v{s}et\v{r}il then obtained a positive answer for an impressive number of structures~\cite{CSS-Ramsey, HN18-closure-operations}.  In~\cite{EvansHubickaNesetril}, Evans, Hubi\v{c}ka and Ne\v{s}et\v{r}il gave an example of an $\omega$-categorical structure without $\omega$-categorical Ramsey expansion, but the problem as formulated in  Question~\ref{quest:Ramsey} remains open.

Given the importance of Ramsey expansions, 
it is natural to 
investigate how their  existence 
is preserved under various constructions.
The stability of the Ramsey property in this sense is the 
main theme of the survey~\cite{BodirskyRamsey}, where one of the claims is that if $\sA$ has a homogeneous Ramsey expansion by first-order definable relations, then so do  its model-complete core~\cite[Theorem~3.18]{BodirskyRamsey} and its  model-companion~\cite[Theorem~3.15]{BodirskyRamsey} (this turns out to be incorrect, as we are going to show). 
The following problems were left open; we note that the definition of a Ramsey structure  in~\cite{BodirskyRamsey}  slightly differs from ours, and that the questions that follow have been reformulated  to agree with our terminology.
\begin{question}[\new{Questions~7.2,~7.3} in~\cite{BodirskyRamsey}]\label{quest:mc-core-Ramsey} Let $\sA$ be a structure.
\begin{enumerate}
    \item Suppose that $\sA$ has a homogeneous Ramsey expansion with a finite signature. Does the model-complete core of $\sA$ (resp., its model-companion) have such an expansion?
    \item Suppose that $\sA$ has an $\omega$-categorical \new{homogeneous} Ramsey expansion. Does the model-complete core of $\sA$ (resp., its model-companion) have such an expansion?
\end{enumerate}
\end{question}

\subsection{Results}\label{sect:results}

Our main theorem is the following.

\begin{theorem}\label{thm:main}
Let $\sA$ be a first-order reduct of an $\omega$-categorical  homogeneous Ramsey structure $\sB$, and let $\sA'$ be its model-complete core. Then:
\begin{itemize}
    \item $\sA'$ also is a first-order reduct of an $\omega$-categorical homogeneous Ramsey structure $\sB'$ that is a substructure of $\sB$.
    \item If $\sB$ is finitely bounded, then $\sB'$ can be chosen to be finitely bounded as well.
\end{itemize}
\end{theorem}

This provides a  positive answer to Question~\ref{quest:mc-core-Ramsey} for model-complete cores. As a by-product, we obtain the same answer for the variant with  model-companions.
Indeed, given any $\omega$-categorical structure $\sA$, let $\sC$ be the expansion of $\sA$ by the complement of each of its relations.
Let $\sC'$ be the model-complete core of $\sC$, and let $\sA'$
be the reduct of $\sC'$ obtained by forgetting the new relations.
Then $\sA'$ is model-complete since it has the same embeddings as $\sC'$. 
Any homomorphism from $\sC$ to $\sC'$ has to be an embedding since
the complement of every relation has to be preserved, so that
there exists an embedding of $\sA$ into $\sA'$, i.e, $\sA'$ is
the model-companion of $\sA$.
Thus, if $\sA$ has a homogeneous Ramsey expansion (with finite signature), so does $\sC$,
which means by Theorem~\ref{thm:main} that $\sC'$ has a homogeneous
Ramsey expansion, and finally this implies that $\sA'$ has
a homogeneous Ramsey expansion.

\new{In the course of this research we discovered that Theorems~3.15 and~3.18 of~\cite{BodirskyRamsey}  mentioned before Question~\ref{quest:mc-core-Ramsey} are incorrect, as shown by the following example. 
Let $E,N, S, <$ be binary relation symbols, and let $P$ be a constant symbol.  Let 
$\sC$ be the Fra\"{i}ss\'{e} limit of the class of finite structures in this signature for which $S$ is the complement of an equivalence relation with two classes, 
$E$ is an undirected graph relation without loops contained in the complement of $S$ and such that the element named by $P$ is not $E$-related to any other element, $N$ is the complement of $E$,   and $<$ is a linear order. Then $\sC$ is homogeneous,  and easily seen to be Ramsey (the square of the automorphism group 
of the random graph with a generic linear order has a continuous homomorphism onto the automorphism group of $\sC$, hence the latter is extremely amenable). Let $\sA$ be the  reduct of $\sC$ with the relations $<$, $E$, $N$, and $S$. Since their automorphism groups are equal, $\sA$ and $\sC$ are first-order reducts of one another. But the model-complete core as well as the model-companion of $\sA$ are isomorphic to the structure which we would have obtained constructing $\sC$ without the special element denoted by $P$; they are homogeneous, but not Ramsey (the latter is a consequence of the fact that their automorphism groups act  continuously on the classes of the complement of $S$ without a fixed point, and hence are not extremely amenable).}

\new{
It follows directly from our theorem by setting $\sA=\sB$ that if $\sA$ is itself homogeneous and Ramsey, then so is its model-complete core (this was   already observed  in~\cite{Bod-New-Ramsey-classes}). One can push this further and use our theorem to derive that if $\sA$ is an $\omega$-categorical structure which has a homogeneous Ramsey expansion by relations with an existential positive definition in $\sA$, then its model-complete core has the same property; this is basically a consequence of the fact that  homomorphisms  preserve existential positive definitions. The same is true for the model-companion and existential definitions. 
These observations  can arguably be viewed as the best possible correction of the aforementioned results from~\cite{BodirskyRamsey}.
}

Since Theorem~\ref{thm:main} is also compatible with the additional condition of finite boundedness, it implies that we obtain a positive answer to Question~\ref{quest:main} provided Question~\ref{quest:Ramsey} has a positive answer: that is, we obtain that the class of first-order reducts of finitely bounded homogeneous Ramsey structures is closed under taking model-complete cores. 

\new{Besides its theoretical interest, Theorem~\ref{thm:main} is also practical in that it gives a concrete way of computing the structures $\sB'$ and $\sA'$. For example,  a corollary of Theorem~\ref{thm:main} is that if $\sA$ is any first-order reduct of the countable universal homogeneous tournament $\mathbb T$ (which is \emph{not} a Ramsey structure itself, but has a homogeneous Ramsey expansion by a linear order), then its model-complete core $\sA'$ is either a one-element structure or is itself a first-order reduct of $\mathbb T$. 
Similar results can be obtained using Theorem~\ref{thm:main} for first-order reducts of countable homogeneous graphs; the (short) proofs of these results can be found in~\cite{SmoothApproximations}. 
Finally, a similar proof shows that if $\sA$ is a any first-order reduct of the countable universal homogeneous partial order $\mathbb P$ and $\sA'$ is its model-complete core, then $\sA'$ is either a one-element structure, a first-order reduct of the order of the rationals, or a first-order reduct of $\mathbb P$. Note that in the second case, if one takes the structure $\sB$  to be the universal homogeneous partial order with a linear expansion in order to apply Theorem~\ref{thm:main}  ($\mathbb P$ does not work since it is not Ramsey), then $\sB'$ is not isomorphic to $\sB$; this case occurs, in particular, if $\sA$ is chosen to be $\mathbb P$ itself.}

It might be interesting to note that our proof does not use Bodirsky's theorem on the existence of model-complete cores. Rather than that, it refines, and perhaps sheds some light on, the second proof of
this result in~\cite{BKOPP-equations}.
On the other hand, it is built upon the Ramsey property, and deriving the  result using our approach seems rather unpractical.



\section{Preliminaries}
We use blackboard boldface letters such as $\sA, \sB, \ldots$ for relational structures, and the same letters in plain font $A, B,\ldots $ for their domains. Similarly,  we write $\fM, \fN,\ldots$ for  transformation semigroups, in particular for permutation groups, and $M,N,\ldots $ for their domains. All domains of relational structures as well as of transformation semigroups are tacitly assumed to be countable.

Let $\fG$ be a permutation group. Then $\fG$ naturally acts componentwise on $G^n$, for all $n\geq 1$. Any minimal non-empty invariant set under any such action will be called an \emph{orbit} of $\fG$; we will sometimes use the notion \emph{$n$-orbit} when we wish to specify the action. A permutation group $\fG$ is \emph{oligomorphic} if it has finitely many $n$-orbits for every $n\geq 1$.

The automorphism group $\Aut(\sA)$ of any relational structure $\sA$  is a permutation group, and $\sA$ is $\omega$-categorical if $\Aut(\sA)$ is oligomorphic. The group $\Aut(\sA)$ is a closed subgroup of the full symmetric group on its domain $A$ with respect to the \emph{topology of pointwise convergence}, i.e., the product topology on $A^A$ where $A$ is taken to be discrete. It is called \emph{extremely amenable} if any continuous action on a compact Hausdorff space has a fixed point.

The endomorphism monoid $\End(\sA)$ of a relational structure $\sA$ is a transformation monoid on its domain $A$. This monoid also acts naturally componentwise on finite powers of $A$, and we shall write $t(a)$ for the $n$-tuple obtained by applying $t\in\End(\sA)$ to a tuple $a\in A^n$, for any $n\geq 1$. The monoid $\End(\sA)$ also bears the topology of pointwise convergence, and $\sA$ is called a \emph{model-complete core} if $\Aut(\sA)$ is dense in $\End(\sA)$ with respect to this topology. If $\fM$ is a subset of $M^M$ (e.g., a  transformation semigroup, or even a permutation group on $M$), then we write $\overline{\fM}$ for the closure of $\fM$ in $M^M$ (and not in the symmetric group on $M$, even if $\fM$ is a permutation group!) with respect to the topology of pointwise convergence. Thus, a function $f\colon M\to M$ is an element of $\overline{\fM}$ if for every finite $F\subseteq M$ there exists $g\in \fM$ such that $g$ agrees with $f$ on $F$. A transformation semigroup $\fM$ is \emph{closed} if $\fM=\overline{\fM}$.

If $\fM,\fN$ are transformation semigroups on the same domain, then $\fM$ is \emph{left-invariant (right-invariant)} under $\fN$ if $n\circ m\in\fM$ ($m\circ n\in\fM$) for all $n\in \fN$ and all $m\in\fM$. The semigroup $\fM$ is \emph{invariant} under $\fN$ if it is left- and right-invariant. A \emph{left-ideal} of $\fM$ is a  subsemigroup which is left-invariant under $\fM$.

Two structures $\sA, \sA'$ are \emph{homomorphically equivalent} if there exists a homomorphism from $\sA$ to $\sA'$ and vice-versa. Whenever $\sA$ is a relational structure, and $\sA', \sA''$ are model-complete cores which are homomorphically equivalent  to $\sA$, then $\sA', \sA''$ are isomorphic. We can thus speak of the model-complete core of a structure $\sA$ if it exists.

A structure $\sA$ is \emph{homogeneous} if every  isomorphism between finite induced substructures of $\sA$ extends to an automorphism of $\sA$.
If $\sA$ is homogeneous, then its \emph{age}, i.e.,
the class of its finite induced substructures up to isomorphism, has the \emph{amalgamation property (AP)}: a class $\cK$ of structures has the AP if for all $\sA_0,\sA_1,\sA_2\in\cK$ and all embeddings $e_1,e_2$ of $\sA_0$ into $\sA_1,\sA_2$, respectively, there exist embeddings $f_1,f_2$ of $\sA_1,\sA_2$ into a structure $\sC\in\cK$ such that $f_1\circ e_1=f_2\circ e_2$.
Conversely, if the age $\cK$ of some countable structure $\sA'$ has the AP,
then there exists a countable homogeneous structure $\sA$ whose age is equal to $\cK$. The structure $\sA$ is called the \emph{Fra\"iss\'e limit of $\cK$}.


A class $\cK$ of finite structures in a common finite signature is called \emph{finitely bounded} if there exists a finite set $\cF$ of structures in that signature such that membership in $\cK$ is equivalent to not embedding any member of $\cF$. An infinite  structure is called finitely bounded if its age is finitely bounded.

Given two structures $\sS,\sF$, an \emph{isomorphic copy of\/ $\sS$ in $\sF$}
is an embedding from $\sS$ to $\sF$.
A structure $\sA$ is \emph{Ramsey} if for all structures $\sS,\sF$ in its age and all colorings of the isomorphic copies of $\sS$ in $\sA$ with two colors there exists an isomorphic copy of $\sF$ in $\sA$ on which the coloring is constant. If $\sA$ is homogeneous, then this is the case if and only if $\Aut(\sA)$ is extremely amenable~\cite{Topo-Dynamics}.


\section{The Proof}

\subsection{Range-rigidity}

\begin{definition}\label{def:rangerigid}
Let $\fG$ be a permutation group, and let $g\colon G\to G$ be a function. We call $g$ \emph{range-rigid} with respect to $\fG$ if for all $\beta\in\fG$ we have  $g\in\overline{ \{\alpha\circ g\circ \beta\circ g\;|\; \alpha\in\fG\}}$; in other words, every orbit of $\fG$ which has a tuple within the range of $g$ is invariant under $g$. 
\end{definition}

\begin{lemma}\label{lem:AP}
Let $\sB$ be a homogeneous structure, and let $g\colon B\to B$ be range-rigid with respect to $\Aut(\sB)$. Then the age of the structure induced by the range of $g$ in $\sB$  has the AP.
\end{lemma}
\begin{proof}
Let $U,V,W\subseteq B$ be so that $g[U], g[V], g[W]$ are finite, and such that the structure $\sS_U$ induced by $g[U]$ embeds into the structures $\sS_V$ and $\sS_W$ induced by $g[V]$ and $g[W]$, respectively. By the homogeneity of $\sB$, we know that these embeddings can be performed by restricting automorphisms $\alpha,\beta\in\Aut(\sB)$ to $g[U]$. Consider the structure $\sF$ induced by $F:=g[\alpha^{-1}[g[V]]\cup \beta^{-1}[g[W]]$ in $\sB$. Then $\sF$ is an amalgam of $\sS_V$ and $\sS_W$ over $\sS_U$ and embeddings given by $\alpha,\beta$: the witnessing embeddings of the amalgamation are the restrictions of $g\circ \alpha^{-1}$ and $g\circ \beta^{-1}$ to $g[V]$ and $g[W]$, respectively:
the fact that these restrictions are embeddings follows from the  range-rigidity of $g$.
\end{proof}

\begin{definition}\label{def:limit}
Let $\sB$ be a  homogeneous structure, and let $g\colon B\to B$ be range-rigid with respect to $\Aut(\sB)$.
\begin{itemize}
    \item We denote by $\sB_g$ the Fra\"{i}ss\'{e} limit of the class of finite structures induced by the range of $g$ (which has the AP, by~Lemma~\ref{lem:AP}). By the  homogeneity of $\sB$, we may assume that  $\sB_g$ is an induced substructure of $\sB$.
    \item If $\sA$ is first-order definable in $\sB$ without quantifiers,  then we denote by $\sA_g$ the substructure induced by the domain of $\sB_g$ in $\sA$.
\end{itemize}
\end{definition}

We remark that in Definition~\ref{def:limit}, if $\sB$ is  $\omega$-categorical, then any first-order reduct $\sA$ has a quantifier-free first-order definition in $\sB$. Hence, $\sA_g$ is \new{well-defined} in that case. 

In the following three lemmas, we show that $\sB_g$ retains the properties of $\sB$ that we are interested in.
\begin{lemma}\label{lem:omegacat}
    Let $\sB$ be a homogeneous structure,
    and let $g\colon B\to B$ be range-rigid with respect to     $\Aut(\sB)$. If\/ $\sB$ is $\omega$-categorical, then so is $\sB_g$.
\end{lemma}
\begin{proof}
    Since $\sB_g$ is homogeneous, it suffices to prove that for every $n\geq1$, there are only finitely many atomic formulas with $n$ free variables modulo equivalence.
    Since $\sB_g$ is a substructure of $\sB$, if two atomic formulas
    are equivalent in $\sB$ then they are equivalent in $\sB_g$.
    Since $\sB$ is homogeneous and $\omega$-categorical, for every $n\geq1$ there are only finitely many atomic formulas with $n$ free variables in $\sB$, and we get the desired result.
\end{proof}

\begin{lemma}\label{lem:finitelybounded}
Let $\sB$ be a homogeneous structure, and let $g\colon B\to B$ be range-rigid with respect to $\Aut(\sB)$. If\/ $\sB$ is finitely bounded, then so is $\sB_g$.
\end{lemma}
\begin{proof}
Let $\cF$ be a finite set of forbidden substructures for the age of $\sB$. Let $m\geq 1$ be the maximum of the arities of the relations of $\sB$. Let $\cF'$ consist of all structures in $\cF$, plus all structures on the set $\{1,\ldots,m\}$ which are isomorphic to a substructure of $\sB$ but not of $\sB_g$. 

Clearly, if a finite structure in the signature of $\sB$ embeds a member of $\cF'$, then it cannot be in the age of $\sB_g$. Conversely, if such a structure $\sF$ does not embed  any member of $\cF'$, then it embeds into $\sB$, hence we may assume it is a substructure of $\sB$. Applying $g$ to this substructure, we obtain a structure isomorphic to $\sF$, because $g$ preserves all $m$-orbits which intersect its range, and all $m$-element substructures of $\sF$ belong to such orbits. Hence, $g$ shows that $\sF$ embeds into $\sB_g$.
\end{proof}

\begin{lemma}\label{lem:Ramsey}
Let $\sB$ be a homogeneous structure, and let $g\colon B\to B$ be range-rigid with respect to $\Aut(\sB)$. If\/ $\sB$ is Ramsey, then so is $\sB_g$.
\end{lemma}
\begin{proof}
Let $\sS,\sF$ be finite induced substructures of $\sB_g$, and let $\chi$ be a coloring of the isomorphic copies of $\sS$ in $\sB_g$ with two colors. Let $f\colon g[B]\to B_g$ be a function which is an embedding with respect to the relations of $\sB$; such an embedding exists by the homogeneity of $\sB_g$ and since the age of the structure induced by $g[B]$ in $\sB$ is equal to the age of $\sB_g$. The coloring $\chi$ then induces a coloring $\chi'$ of the isomorphic copies of $\sS$ in $\sB$: any such copy is by composing with  $f\circ g$ sent  to an isomorphic copy of $\sS$ in $\sB_g$, by the range-rigidity of $g$; this second copy is assigned a color by $\chi$.  Since $\sB$ is Ramsey, there exists an isomorphic copy of $\sF$ in $\sB$ on which $\chi'$ is constant. Composing this copy with  $f\circ g$ then yields, again by the range-rigidity of $g$, a copy of $\sF$ in $\sB_g$. Clearly, $\chi$ is constant on this latter copy.
\end{proof}

\begin{lemma}\label{lem:core}
Let $\sB$ be a  homogeneous structure, and let $\sA$ be first-order definable in $\sB$ without quantifiers. Suppose that $\fN\subseteq \End(\sA)$ is a minimal closed left ideal, and that  $g\in \fN$ is  range-rigid with respect to $\Aut(\sB)$. Then $\sA_g$ is the   model-complete core of $\sA$.
\end{lemma}
\begin{proof}
We first prove that the set 
\[\mathcal I:=\{(a,b) \mid \exists n\geq 1\;  \exists e\in\End(\sA_g)\;\; (a,b\in (A_g)^n \wedge   e(a)=b) \} \]
is a back-and-forth system of partial isomorphisms of $\sA_g$.
This implies that $\sA_g$ is a model-complete core,
since we then obtain for every $e\in\End(\sA_g)$ and every finite tuple $a$
an automorphism $\alpha\in\Aut(\sA_g)$ such that $\alpha(a)=e(a)$. To prove that $\mathcal I$ consists of partial isomorphisms of $\sA_g$
and has the back-and-forth property, note that it suffices to show the following: for every $e\in\End(\sA_g)$ and every finite tuple $a$ of elements in $\sA_g$, there exists $e'\in\End(\sA_g)$ such that $e'(e(a))=a$.

To this end, let $f\colon g[B]\to B_g$ be a function which is an embedding with respect to the relations of $\sB$,
such an embedding exists since the age of the structure induced by $g[B]$ in $\sB$ is the same as the age of $\sB_g$.
\new{Moreover, one can assume that $a$ lies within the range of $f\circ g^2$.
Indeed, also the age of the structure induced by $g^2[B]$ in $\sB$ is the same as that of $\sB_g$. Thus, there is a tuple $b$ in $g^2[B]$ which induces the same substructure as $a$.
Since $f$ is an embedding with respect to the relations of $\sB$, the structures induced by $f(b)$ and $a$ in $\sB_g$ are isomorphic, so that by the homogeneity of $\sB_g$ there is an automorphism $\alpha\in\Aut(\sB_g)$ such that $\alpha(f(b))=a$.
Replacing $f$ by $\alpha\circ f$, we can therefore assume without loss of generality that $a$ lies in the range of $f\circ g^2$.}
Let $b$ be a tuple such that $a=(f\circ g^2)(b)$. By the fact that $g$ lies in a minimal closed left ideal of $\End(\sA)$, there exists $h\in\End(\sA)$ such that $$
h\circ (e\circ f\circ g)(g(b))=g(b)\; .
$$ 
Hence,
$$
a=(f\circ g^2)(b)=(f\circ g)(g(b))=(f\circ g)\circ h\circ (e\circ f\circ g)(g(b))=(f\circ g\circ h)\circ e(a)\; .
$$
The restriction of the function $f\circ g\circ h\in\End(\sA)$ to $A_g$ thus bears witness to the statement we wanted to prove.

Since $\sA_g$ is a substructure of $\sA$, and since $f\circ g$ is a homomorphism from $\sA$ to $\sA_g$, the structures $\sA_g$ and $\sA$ are homomorphically equivalent. Whence, $\sA_g$ is indeed the model-complete core of $\sA$.
\end{proof}

\subsection{From extreme amenability to canonicity to range-rigidity}
The following definition from~\cite{BPT-decidability-of-definability} is incomparable to range-rigidity, but will allow us to produce range-rigid functions.
\begin{definition}\label{def:canonical}
Let $\fG$ be a permutation group, and let $g\colon G\to G$ be a function. We call $g$  \emph{canonical} with respect to $\fG$ if for all $\beta\in\fG$ we have $g\in\overline{ \{\alpha\circ g\circ \beta\;|\; \alpha\in\fG\}}$; in other words, the image of any orbit of $\fG$ under $g$ is contained in an orbit.
\end{definition}

\begin{lemma}[\new{The canonisation lemma (Lemma 14 in~\cite{BPT-decidability-of-definability}, Theorem 5 in~\cite{BodPin-CanonicalFunctions})}]\label{lem:canonisation}
Let $\fG$ be a closed oligomorphic extremely amenable permutation group, and let $g\colon B\to B$. Then 
$$
\overline{\{\alpha\circ g\circ \beta \;|\; \alpha,\beta\in\fG\}}
$$
contains a canonical function with respect to $\fG$.
\end{lemma}

\begin{lemma}\label{lem:range-rigid}
Let $\fG$ be a closed oligomorphic permutation group, and let $\fM\subseteq G^G$ be a non-empty closed transformation semigroup which is invariant under $\fG$ and which contains a canonical function with respect to $\fG$. Then $\fM$ contains a range-rigid  function with respect to $\fG$.
\end{lemma}
\begin{proof}
Pick any canonical function $g\in \fM$.  For every $n\geq 1$, there exist $k_n\geq 1$ such that $g^{k_n}\circ g^{k_n}[O]\subseteq g^{k_n}[O]$ for all $n$-orbits $O$ of $\fG$.
Let $\sim$ be the equivalence relation on $G^G$ defined by $f\sim f'$ if
$f'\in \overline{\{\alpha\circ f\;|\; \alpha\in\fG\}}$. It is known that $G^G/\sim$ is compact (see, e.g., Lemma~4 in~\cite{BodPin-CanonicalFunctions}).
The sequence $([g^{k_n}]_\sim)_{n\geq 1}$ thus has an accumulation point
in $G^G/\sim$, which means that there exists a sequence $(\alpha_n)_{n\geq 1}$ of elements of $\fG$ such that $(\alpha_n\circ g^{k_n})_{n\geq 1}$ has an accumulation point $h$ in $G^G$,
and \new{$h\in\fM$ since $\fM$ is closed}.
We then have that $h^2[O]\subseteq h[O]$ for all orbits $O$ of $\fG$. Since $h$ is also canonical with respect to $\fG$, this means that $h$ is range-rigid with respect to $\fG$.
\end{proof}

\begin{lemma}\label{lem:betterleftideal}
Let $\fM$ be a closed  transformation monoid containing a closed oligomorphic extremely amenable permutation group $\fG$. Then $\fM$ contains a minimal closed left ideal which contains a range-rigid function with respect to $\fG$.
\end{lemma}
\begin{proof}
The fact that $\fM$ contains a minimal closed left ideal $\fN$ can be proved by a standard compactness argument (see~\cite{BKOPP, BKOPP-equations} for the proof).

Pick any $g\in \fN$, and let $\fS$ be the smallest non-empty closed transformation semigroup which contains $g$ and which is  invariant under $\fG$. By Lemma~\ref{lem:canonisation}, $\fS$ contains a canonical function with respect to $\fG$, and hence by Lemma~\ref{lem:range-rigid}, $\fS$ contains a range-rigid function with respect to $\fG$. It is easy to see that any element of $\fS$ belongs to a minimal closed left ideal of $\fM$.
\end{proof}

\subsection{Summary of the proof of Theorem~\ref{thm:main}}
\begin{proof}[Proof of Theorem~\ref{thm:main}] Applying Lemma~\ref{lem:betterleftideal} to the monoid $\fM:=\End(\sA)$ and the group $\fG:=\Aut(\sB)$, we get that $\End(\sA)$ contains a minimal closed left ideal which contains a range-rigid function $g$ with respect to $\Aut(\sB)$. By the $\omega$-categoricity and the homogeneity of $\sB$, the first-order reduct $\sA$ has a quantifier-free definition in $\sB$. Hence, $\sA_g$ is well-defined, and by Lemma~\ref{lem:core}, $\sA_g$ is the model-complete core of $\sA$. By definition, $\sA_g$ has a quantifier-free definition in $\sB_g$, and in particular is a first-order reduct thereof. Lemma~\ref{lem:omegacat} gives that $\sB_g$ is $\omega$-categorical, while Lemma~\ref{lem:Ramsey} tells us that $\sB_g$ is a homogeneous Ramsey structure.
Finally, by Lemma~\ref{lem:finitelybounded} it is finitely bounded if $\sB$ is.
\end{proof}

\ignore{
As announced in Section~\ref{sect:results}, our techniques can be adapted to give another proof of Theorem~\ref{thm:cores}.
We only give an outline here. The main idea is so define range-rigidity for a transformation semigroup $\fS$.
We define an \emph{orbit} of $n$-tuples to be a minimal non-empty set $X$ such that for all $a,b\in X$, there exist $t,t'\in\fS$ such that $t(a)=b$ and $t'(b)=a$.
A structure $\sA$ is \emph{weakly oligomorphic} if $\End(\sA)$
has finitely many orbits of $n$-tuples, for all $n\geq 1$.
Say that a structure $\sB$ is \emph{weakly homogeneous} if every partial isomorphism between finite substructures of $\sB$ extends to a self-embedding of $\sB$. Then it is easily seen that Lemmas~\ref{lem:AP}, \ref{lem:omegacat}, \ref{lem:core} hold when $\sB$ is assumed to be weakly homogeneous, and $g$ is range-rigid with respect to $\Emb(\sB)$.
Moreover, if $\fN$ is any minimal closed left ideal of $\Emb(\sB)$,
then any $g\in\fN$ is range-rigid with respect to $\Emb(\sB)$.
Finally, it suffices to notice that if $\sA$ is any weakly oligomorphic structure, then the structure $\sB$ whose relations are given by the orbits of $\End(\sA)$ is weakly homogeneous.}

\bibliographystyle{plain}

\bibliography{cores.bib}

\end{document}